\documentclass{article}

\usepackage{lineno,hyperref}
\usepackage{amsfonts}
\usepackage{amsthm}

\usepackage{graphicx}
\usepackage{tikz}
\usepackage{tikz-qtree}
\usepackage{qtree}

\usepackage{tikz-cd,mathtools}
\tikzset{mydescription/.style={anchor=center,fill=white}}

\usepackage{eufrak}
\usepackage{xcolor}

\newtheorem{theorem}{Theorem}
\newtheorem{lemma}[theorem]{Lemma}
\newtheorem{proposition}[theorem]{Proposition}
\newtheorem{corol}[theorem]{Corollary}
\newtheorem*{conjecture}
{Conjecture}
\newtheorem{remark}[theorem]{Remark}

\newtheorem{definition}[theorem]{Definition}

\newcommand{\N}{{\mathbb N}}

\newcommand{\Cz}{{\mathcal{C}\mathfrak{z}}}
\newcommand{\SN}{{\mathcal S}({\mathbb N})}
\newcommand{\IN}{{\mathcal I (\mathbb N)}}
\begin{document}


\title{From a conjecture of Collatz to Thompson's \\ group $\mathcal F$,  via a conjunction of Girard}

\author{Peter M. Hines}




\maketitle

\begin{abstract}	
	The famous $3x+1$ problem of L. Collatz needs no introduction; however, this paper concerns a lesser-known, but similarly unresolved, precursor problem : the Original Collatz Conjecture, or OCC \cite{Lag85}. 
	
	We demonstrate that the core arithmetic operator from the OCC, when combined with a conjunction of J.-Y. Girard from his Geometry of Interaction system \cite{GOI1,GOI2}, leads to a realisation of R. Thompson's group $\mathcal F$ as congruential functions, in the sense of J. Conway \cite{JC72}.
	
	We also give the underlying category theory that accounts for this, and describe the core operator from the OCC as a canonical coherence isomorphism. 
\end{abstract}

\begin{center}
\end{center}



\section{Basic Definitions : the OCC, and Girard's conjunction}
We start with the Original Collatz Conjecture, as described in \cite{Lag85}.
\begin{definition} 
	We denote the symmetric group on $\mathbb N$ by $\SN$, and define the {\bf original Collatz bijection} $\rho\in \SN$ 
	by $\rho (n) = \ \left\{ \begin{array}{lcr}
	\vspace{0.3em}
	\frac{2n}{3} & & n\ (mod\ 3) = 0, \\ 
	\vspace{0.3em}
	\frac{4n-1}{3} & & n \ (mod \ 3) = 1, \\ 
	\vspace{0.3em}
	\frac{4n+1}{3} & & n \ (mod \ 3) = 2. 
	\end{array}\right.$
\end{definition}
\begin{lemma}
	The unique fixed points of $\rho$ are $\{ 0,1\}\subseteq \mathbb N$, and the orbits of all $\{ 0,\ldots , 9 \} \setminus \{ 8\}$ are finite.
\end{lemma}
\begin{proof} Elementary algebra gives $\{0,1\}$ as unique solutions to $\rho(n)=n$. By direct calculation, $2\rightarrow 3 \rightarrow 2$ and $4\rightarrow 5\rightarrow 7\rightarrow 9\rightarrow 6\rightarrow 4$ are finite orbits.
\end{proof} 
As described in \cite{Lag85}, the above bijection appeared in unpublished notebooks of Collatz, dated $1^{st}$ May, 1932, as the core of the following (currently unresolved) conjecture : 
\begin{conjecture}[The O.C.C.]  The orbit of $8$ under $\rho$ is infinite.
\end{conjecture}

We combine the core operator of the OCC with the model of a logical conjunction used by J.-Y. Girard \cite{GOI1,GOI2}. As pointed out in \cite{PHD,HA,AHS}, Girard's system was based around the inverse monoid $\mathcal I(\mathbb N)$ of partial injections on $\mathbb N$; we restrict ourselves to the subgroup $\SN\subseteq \mathcal I(\mathbb N)$.

\begin{definition} {\bf Girard's conjunction} is the function $\_\star \_ : \SN\times \SN\rightarrow \SN$ given by 
	$(a\star b)(n) = \left\{ \begin{array}{lcr}
	\vspace{0.3em}
	2.a\left( \frac{n}{2}\right) & \ \ & n \ \mbox{even,} \\
	2.b\left( \frac{n-1}{2}\right)+1 & \ \ & n \ \mbox{odd.} \\
	\end{array} \right.$\\
	It is well-established (e.g. \cite{PHD,HA,AHS}) that this is an injective homomorphism. 
\end{definition}

\subsection{A slight simplification}
The OCC considers orbits of the natural numbers under the bijection $\rho$. We conjugate $\rho$ by the successor function / its (partially defined) inverse to eliminate the triviality $\rho(0)=0$, but preserve the behaviour  of all other orbits : 
\begin{definition}\label{rcb-def} We define the {\bf reduced Collatz bijection} $\lambda\in \SN$ by $\lambda(n)=\rho(n+1)-1$. This is a well-defined  bijection on $\mathbb N$, since $\rho(0)=0$. Expanding out the definition gives an explicit formula :  
	\[ \lambda (n) = \ \left\{ \begin{array}{lcr}
	\vspace{0.3em}
	\frac{4n}{3} & & n\ (mod\ 3) = 0, \\ \vspace{0.3em}
	\frac{4n+2}{3} & & n \ (mod \ 3) = 1, \\ \vspace{0.3em}
	\frac{2n-1}{3} & & n \ (mod \ 3) = 2. 
	\end{array}\right.
	\]
	By construction, $ \lambda^K(n)=\rho^K(n+1)-1$, so questions about the non-trivial orbits of $\rho$ and those of $\lambda$ are interchangeable; the OCC then becomes the conjecture that the orbit of $7$ under $\lambda$ is infinite. 
\end{definition} 

\begin{lemma}\label{noCoincidence-lem} $\lambda\neq\rho$, and $\lambda(n)=\rho(n) \ \Leftrightarrow \ n=0$, for all $n\in \mathbb N$.
\end{lemma}
\begin{proof} Trivially, $\lambda\neq\rho$. Writing the three cases separately gives  
	\[ \lambda(n)\ = \ \left\{ 
	\begin{array}{lcr} 
	4n/3 & \ \ \ \ n=3M \ \ \ \ & 2n/3  \\
	(4n+2)/3 & \ \ \ \ n=3M+1 \ \ \ \ & (4n-1)/3\\
	(2n-1)/3 & \ \ \ \ n=3M+2 \ \ \ \  & (4n+1)/3 \\
	\end{array}\right\} \ = \ \rho(n) \]
	giving, respectively, $n=0$, a contradiction, and  $n=-1\notin \N$.
\end{proof}
\subsection{Combining the original and reduced Collatz bijections}
The inverses of $\lambda,\rho\in \SN$ may be given explicitly as : 
\[ 
\lambda ^{-1}(n) \ = \  
\left\{ 
\begin{array}{lr}  
\vspace{0.3em}
\frac{3n}{4} & n \ (mod \ 4) = 0, \\ \vspace{0.3em}
\frac{3n-2}{4} & n \ (mod \ 4)=2, \\ 
\frac{3n+1}{2} &  n \ \mbox{ odd,}
\end{array} 
\right. 
\  \mbox{ and }  \ 
\rho^{-1} (n) = \ \left\{ \begin{array}{lcr}
\vspace{0.3em}
\frac{3n}{2} & & n\ (mod\ 2) =0, \\ \vspace{0.3em}
\frac{3n+1}{4} & & n \ (mod \ 4) = 1, \\ 
\frac{3n-1}{4} & & n \ (mod \ 4) = 3. 
\end{array}\right.
\]
Composing the reduced Collatz bijection with the inverse of the original Collatz bijection results in a bijection previously seen in category theory \& logic : 
\begin{definition}\label{assoc-def}
	We define the {\bf associator} $\alpha\in \SN$  by $\alpha=\lambda\rho^{-1}\in \SN$; equivalently, 
	$\alpha(n) = \rho(\rho^{-1}(n) +1)-1$. Expanding out the definition gives : 
	\[ \alpha (n) \ = \ 
	\left\{ \begin{array}{lr} \vspace{0.3em} 
	2n 				&  n\ (mod \ 2)=0, \\ \vspace{0.3em}
	n+1 			& n \ (mod \ 4) =1, \\  
	\frac{n-1}{2} 	& n \ (mod \ 4) =3, \end{array} \right. 
	\ \ \mbox{ and } \ \  
	\alpha^{-1}(n) \ = \ \left\{ \begin{array}{lr}
	\vspace{0.3em}
	\frac{n}{2}   & n \ (mod \ 4) = 0, \\  \vspace{0.3em}
	n-1  & n \ (mod \ 4) = 2, \\  
	2n+1   & n \ (mod \ 2) = 1. 
	\end{array}\right.
	\] 
\end{definition}
We may think of the associator as, {\em ``The shifted Collatz bijection, composed with the inverse of the original bijection''}.  
By contrast with either Collatz bijection, orbits of natural numbers under the associator are very easy to characterise.

\begin{lemma}\label{assocOrbits-lem}
	The associator has a unique fixed point, $\alpha(0)=0$. The orbit of every $n>0\in\mathbb N$ has a single local minimum iff $n>0$ is even; no orbit has a local maximum, and thus the orbit of every $n>0$ is infinite.
\end{lemma}
\begin{proof} {\em It is easier (but equivalent) to prove this for the inverse, $\alpha^{-1}=\rho\lambda^{-1}$.}
	
	Trivially, $\alpha^{-1}(0)=0$; this is unique, since $\alpha^{-1}(n)=n \ \Leftrightarrow \ r(n)=l(n)$, contradicting Lemma \ref{noCoincidence-lem} for all $n\neq 0$. 
	When $n\in \mathbb N$ is odd, $\alpha^{-1}(n)= 2n+1>n$; which  is similarly odd. Thus, $n< \alpha^{-1}(n)$ for any odd $n$. When $n>0$ is even, $\alpha^{-1}(n)<n$ may be even or odd, but cannot be zero since $\alpha^{-1}$ is a bijection. Thus, any trajectory starting with a non-zero even number decreases until it arrives at an odd number (the local minimum), at which point it increases without limit.
	
	Every finite orbit attains a maximum and a minimum; thus the only finite orbit under $\alpha^{-1}$ (and hence under $\alpha$) is the fixed point $\alpha^{-1}(0)=0=\alpha(0)$.
\end{proof}
Girard's conjunction and the associator (\& hence the original Collatz bijecton) are linked by a property more familiar from category theory : 
\begin{lemma}[Naturality]\label{nat-lemma} $\alpha (f\star (g\star h)) \ = \ ((f\star g)\star h) \alpha$, for all $f,g,h\in\SN$.
	
\end{lemma}
\begin{proof}
	Direct calculation gives, for all $n\in \N$,
	\[ \alpha (f\star (g\star h)) (n) \ = \ ((f\star g)\star h) \alpha (n) \ = \ \left\{ \begin{array}{lcr}
	4f\left(\frac{n}{2}\right) & \ \ & n\ (mod \ 2) =0 \\
	4g\left(\frac{n-1}{4}\right)+2 & \ \ & n\ (mod \ 4) =1 \\
	2h\left(\frac{n-3}{4}\right)-1 & \ \ & n\ (mod \ 4) =3 \\
	\end{array}\right. 
	\]
\end{proof}

\section{Thompson's $\mathcal F$ from Collatz and Girard}
A significant object of study in combinatorial group theory is Richard Thompson's group $\mathcal F$; we refer to \cite{CFP,MB96} for an overview. Many concrete realisations are known; we start with an abstract definition as generators and relations.
\begin{definition}\label{F-def}
	Thompson's group $\mathcal F$ is generated by the set $\{ x_j \}_{j\in \mathbb N}$ subject to the relations 
	$x_i^{-1} x_j x_i = x_{j+1}$, for all $i<j$. It is well-known (e.g. \cite{CFP}) that this is not a minimal generating set; the subset $\{x_0,x_1\}$ generates the entirety of $\mathcal F$, but the required relations are less intuitive.
\end{definition}

\begin{theorem} The group generated by $\{ \alpha \ ,\ Id\star \alpha \}$ is isomorphic to $\mathcal F$.
\end{theorem} 
\begin{proof}
	We define $\{ X_k \}_{k\in \mathbb N}\subseteq \SN$ by $X_0=\alpha$ and $X_{j+1}=Id \star X_j$, for all $j>0$. 
	As $(\_ \star \_)$ is a group homomorphism, Lemma \ref{nat-lemma} (naturality) implies 
	\[ \alpha (Id\star (Id\star f)) \alpha^{-1} \ =\ (Id\star Id) \star f \ = \ Id \star f \ \ \ \ \forall f\in \SN\]  
	Applying this to the inductive definition of $X_j$ gives $X_j = X_i X_{j+1}X_i^{-1}$, for all $i<j$, and so the subgroup of $\SN$ generated by $\{ X_j\}_{j\in \mathbb N}$ is a homomorphic image of $\mathcal F$, given by $x_j\mapsto X_j$, for all $j\in \mathbb N$. We then appeal to the well-known  property (.e.g. \cite{CFP}) that every such homomorphic image is either abelian or isomorphic to $\mathcal F$ itself, to deduce that this subgroup is isomorphic to $\mathcal F$.
	
	Finally, the presentation of Definition \ref{F-def} is not minimal and $\{ x_0,x_1 \}$ suffices to generate the whole of $\mathcal F$;  our result follows.
\end{proof}

\begin{definition} We refer to the subgroup of $\SN$ generated by 
	\[ \alpha (n) \ = \ \left\{ \begin{array}{lr} \vspace{0.3em} {2n} & { n} \ (mod \ 2) = 0 \\ \vspace{0.3em}
	n+1 & n \ (mod \ 4) =1 \\  
	\frac{n-1}{2} & n \ (mod \ 4) =3 \end{array} \right.  \ \  \ \ (Id\star \alpha)
	(n) \ = \  \left\{ \begin{array}{lr}  \vspace{0.3em}
	n & n \ (mod \ 2) = 0 \\ \vspace{0.3em}
	2n-1 & n \ (mod \ 4)=1 \\ \vspace{0.3em}
	n+2 & n \ (mod \ 8) =3 \\
	\frac{n-1}{2} & n \ (mod \ 8) =7 \end{array} \right. \]
	as the {\bf Congruential realisation of $\mathcal F$}.

	We may  think of $(Id\star \alpha)$ as, `replicating $\alpha=\rho \left( \rho^{-1}(n)+1\right)-1$ on the odd numbers only'. By construction, we may write the generators of this realisation of $\mathcal F$ in terms of the original Collatz bijection, as  :
	\[ n\ \mapsto \   \rho \left( \rho^{-1}(n)+1\right)-1 \ \ \mbox{and} \ \ n \mapsto \   \left\{ \begin{array}{lcr} \vspace{0.3em} n & & n \ \mbox{even,} \\ 2 . \rho \left( \rho^{-1} \left( \frac{n-1}{2} \right) +1 \right) -1 & \ \ & n \ \mbox{ odd.} \end{array} \right. \]
\end{definition}
We observe that the above realisation of $\mathcal F$ may also be derived from that given by M.V. Lawson in \cite{MVL06} in terms of the polycyclic monoids of inverse semigroup theory \cite{MVL}, by relying on a common \& well-known arithmetic representation of these inverse monoids \cite{NP}.

\begin{remark}[Girard's conjunction on $\mathcal F$]
	In \cite{KB}, K. Brown describes an injective group homomorphism $\mu : \mathcal F \times \mathcal F \rightarrow \mathcal F$ that is ``associative up to conjugation by {\em [the generator]} $x_0$'', so 
	\[ \mu(\mu(a,b),c) \ = \ x_0 \mu(a , \mu(b,c)) x_0^{-1} \ \ \ \forall \ a,b,c\in \mathcal F \]
	This is of course the naturality property of Lemma \ref{nat-lemma}, and the generator $x_0$ in the congruential realisation is precisely the associator, derived from the original Collatz bijection. We may therefore identify Girard's conjunction (when restricted to this realisation of $\mathcal F$) with Brown's homomorphism, and conclude that the congruential realisation of $\mathcal F$ is closed under Girard's conjunction.
\end{remark}

\subsection{$\mathcal F$ as congruential functions}
The above terminology is based on the following definition of J. Conway : 
\begin{definition}\label{cf-def}
	A function $f:\mathbb N \rightarrow \mathbb N$ is {\bf congruential} when there exists some indexed set $\{ (x_j,y_j) :x_j,y_j<K \}_{j=0,\ldots,K-1}$ such that $f(Km+i) \ = \ x_jm+y_j$. 
\end{definition}
Conway's key result \cite{JC72} (see also \cite{SM}) -- that universal computation may be expressed via simple iterative problems on congruential functions -- was later refined \& simplified by S. Burckel \cite{SB}, who used a strict subset of Conway's functions.

It may readily be verified that the bijections $\rho, \lambda$, and $\alpha$ are congruential, and this property is preserved under composition, inverses, and Girard's conjunction. The above realisation of $\mathcal F$ is indeed then a realisation as congruential functions. 

\begin{remark}It remains to describe the structure of the group of congruential bijections generated by $\{ \lambda,\rho,(Id\star \lambda),(Id\star\rho)\}$. This is presumably non-trivial, but  worthwhile; it arises directly from the O.C.C. and contains Thompson's $\mathcal F$  in a natural way. Some identities this group satisfies may be derived categorically, from Figure \ref{pentagram-fig}. 
\end{remark} 

\subsection{An explanation}
The explanation why such seemingly distinct topics (Collatz's conjecture(s), Girard's linear logic, Thompson's $\mathcal F$) are linked is  categorical. All the key arithmetic operations (Collatz's bijection, the reduced Collatz bijection, \& the associator) are {\em canonical isomorphisms}, in the sense of categorical coherence.  Further, Girard's conjunction (or Brown's homomorphism) is a categorical tensor, and Thompson's $\mathcal F$ has been known since its introduction to have close connections with MacLane's celebrated coherence theorem for associativity\footnote{From \cite{PD96}, {\em ``The only non-trivial relations [of $\mathcal F$] correspond to the well-known MacLane-Stasheff pentagon''}.}.

We now describe and discuss the structures for which the bijections we have studied mediate coherence. 

\section{An interpretation as categorical coherence}\label{cats-sect}
\noindent{\em (We assume familiarity with the basics of category theory, MacLane's coherence theorem for associativity \cite{MCL,GMK}, \& the relationship between coherence and  Stasheff's associahedra \cite{JLL}.)}
\\

It is well-known, at least in some fields of category theory, that Thompson's group $\mathcal F$ describes coherence for associativity in the following setting :  

\begin{definition}
	A {\bf semi-monoidal category} $(\mathcal C,\otimes)$ is required to satisfy all the MacLane / Kelly axioms for a monoidal category except for the existence of a unit object, so there exists a natural isomorphism between the functors 
	\[ (\_ \otimes (\_ \otimes \_)) \ \ , \ \ ((\_ \otimes \_ ) \otimes \_ ) \ : \ \mathcal C \times \mathcal C \times \mathcal C \rightarrow \mathcal C \]
	whose components $\{ \tau_{A,B,C} : A\otimes (B \otimes C) \rightarrow (A\otimes B) \otimes C \}_{A,B,C\in Ob(\mathcal C)}$ (the {\bf object-indexed associators}) satisfy MacLane's {\bf pentagon condition}
	\[  ( \  \tau_{A,B,C}\otimes Id_D)  \  \tau_{A,B\otimes C,D} \  (Id_A \otimes    \tau_{B,C,D}) \ = \    \tau_{A\otimes B,C,D}  \  \tau_{A,B,C\otimes D} \]
	A single-object category (i.e. a monoid) $\mathcal M$ may be semi-monoidal, without its unique object being required to be a unit object. In this setting, the required natural transformation will have a unique component (its {\bf associator}) $\tau\in \mathcal M$, and MacLane's pentagon condition becomes $\tau^2=(\tau\otimes Id)\tau(Id\otimes \tau)$.
\end{definition}

We refer to \cite{K,JK} for the theory of semi-monoidal categories, and \cite{TAC,JHRS} for the monoid-theoretic setting.  The following is well-established : 
\begin{theorem}\label{noSS-thm} Let $(M,\star)$ be a semi-monoidal monoid for which the functors (homomorphisms) $(Id_M\star \_),(\_ \star Id_M):M\rightarrow M$ are faithful. Then either \begin{enumerate}
		\item the associator for $\_\star\_$ is the identity $Id\in M$, in which case the unique object of $M$ is (a retract of) the unit object for $\_\star\_$, or 
		\item the set of all canonical associativity isomorphisms is a group isomorphic to Thompson's $\mathcal F$.
	\end{enumerate}
\end{theorem}
\begin{proof} Part 1. was given in \cite{JHRS}. Part 2. was observed independently by many researchers. A historical account together with a proof (with no claim to originality) is given in \cite{OCL}. Key contributions include, but are not restricted to, \cite{MVL06,JHRS, FL,MB96,PD96}; not all of these are phrased categorically.
\end{proof}

\begin{theorem} Girard's conjunction $\_ \star \_$ is a tensor on $\SN$, whose unique associativity isomorphism is the associator $\alpha\in\SN$  of Definition \ref{assoc-def}.
\end{theorem}
\begin{proof} This is well-established (see \cite{PHD,AHS}) with explicit algebraic / modular arithmetic formul\ae\ given in \cite{RC}. Note that none of these references observed the decomposition into the bijection from the Original Collatz Conjecture.\end{proof}
As a corollary, we may give the relationship between Girard's conjunction and the associator as a pentagonal commuting diagram (MacLane's pentagon) -- the subdiagram of Figure \ref{pentagram-fig} shown in {\color{red} red}; the required identities for this sub-diagram are also derived algebraically in \cite{RC}.

\begin{remark} We observe that Lemma \ref{noCoincidence-lem} may be derived categorically from Theorem \ref{noSS-thm}; the identity $\lambda=\rho$ would imply that Girard's conjunction was a strictly associative tensor on $\SN$, and thus that $\SN$ was a commutative monoid -- a contradiction!\end{remark}

Instead, we may use the decomposition of this associator in terms of Collatz's bijection, $\alpha=\lambda\rho^{-1}$, to factor each edge of MacLane's pentagon; relying on the functoriality of the tensor, this gives the additional edges shown in {\color{olive} green}. By construction, the sub-diagram with {\color{red} red} and {\color{olive} green} edges commutes. 

We give a categorical account of this decomposition:
\begin{definition} \label{star3-def} We define the 
	homomorphism $(\_ \star \_ \star \_):\SN^{\times 3}\rightarrow \SN$ by 
	\[ (f\star g\star h)(n) \ = \ \left\{\begin{array}{lcr} 
	\vspace{0.3em}
	3f\left( \frac{n}{3} \right) & \ \ & n \ (mod \ 3) = 0 \\
	\vspace{0.3em}
	3g\left( \frac{n-1}{3} \right)+1 & \ \ & n \ (mod \ 3) = 1 \\
	3h\left( \frac{n-2}{3} \right)+2 & \ \ & n \ (mod \ 3) = 2 \\
	\end{array}\right.
	\]
\end{definition}

\begin{proposition}
	There exist natural isomorphisms
	\begin{enumerate}
		\item $(\_ \star \_ \star \_ ) \Rightarrow ( (\_ \star \_) \star\_ )$
		\item $(\_ \star \_ \star \_ ) \Rightarrow ( \_ \star (\_ \star\_ ))$
	\end{enumerate}
	whose unique components are 
	\begin{enumerate}
		\item The reduced Collatz bijection $\lambda\in \SN$
		\item The original Collatz bijection $\rho\in \SN$ 
	\end{enumerate}
	respectively.
\end{proposition}
\begin{proof} This follows by direct calculation; we observe that 
	\[ \rho(f\star g\star h)(n) = (f\star(g\star h))\rho^{-1}(n) = \left\{\begin{array}{lcr} 
	\vspace{0.3em}
	2f\left( \frac{n}{3} \right) & \ \ & n \ (mod \ 3) = 0 \\
	\vspace{0.3em}
	4g\left( \frac{n-1}{3} \right)+1 & \ \ & n \ (mod \ 3) = 1 \\
	4h\left( \frac{n-2}{3} \right)+3 & \ \ & n \ (mod \ 3) = 2 \\
	\end{array}\right.
	\]
	and similarly
	\[ \lambda(f\star g\star h)(n) = ((f\star g)\star h)\lambda^{-1}(n) = \left\{\begin{array}{lcr} 
	\vspace{0.3em}
	4f\left( \frac{n}{3} \right) & \ \ & n \ (mod \ 3) = 0 \\
	\vspace{0.3em}
	4g\left( \frac{n-1}{3} \right)+2 & \ \ & n \ (mod \ 3) = 1 \\
	2h\left( \frac{n-2}{3} \right)+1 & \ \ & n \ (mod \ 3) = 2 \\
	\end{array}\right.
	\]
\end{proof}

\begin{definition}
	We denote by $({\bf Grp},\times )$ the (strictly associative) monoidal category of groups \& homomorphisms with Cartesian product, and denote the endomorphism operad of $\SN$ by $Endo(\SN)$.
\end{definition}

\begin{lemma}\label{freeOperad-lem}
	Both $(\_\star \_):\SN^{\times 2}\rightarrow \SN$ and $(\_\star \_\star\_):\SN^{\times 3}\rightarrow \SN$ are operations in $Endo(\SN)$, and the sub-operad they generate is freely generated.
\end{lemma}
\begin{proof} For operations up to arity 5, this may be verified by direct calculation. Injectivity and the natural induction argument then gives the general case.
\end{proof}

\begin{corol}
	The component of the natural isomorphism from $(\_ \star (\_ \star \_ ))$ to $((\_ \star \_ )\star \_ )$ is given by the composition of the components of :
	\begin{itemize}
		\item the natural isomorphism from $(\_ \star (\_ \star \_ ))$ to $(\_ \star \_ \star \_ )$
		\item the natural isomorphism from $(\_ \star \_ \star \_ )$ to $((\_ \star \_) \star \_ )$
	\end{itemize}
	and hence $\alpha = \lambda\rho^{-1}$.
\end{corol}

We may also add in edges between the `inner triangles' of Figure \ref{pentagram-fig}, simply by taking composites along the  {\color{olive} green} paths, to give the (commuting) pentagon shown in {\color{blue} blue}.  Either by construction, by direct calculation, or as a Corollary of Lemma \ref{freeOperad-lem}, the entire diagram, of {\color{red} red}, {\color{olive} green}\ and {\color{blue} blue} edges then commutes.

\begin{figure}[h]\caption{A Commuting Pentagram over $\SN$}\label{pentagram-fig}
	\begin{center}	
		\scalebox{0.85}{
			\begin{tikzpicture}[>=latex]
			\tikzset{every node/.style={scale=.8},font=\large}
			\def\radius{5.3cm} 
			\def\offset{18} 
			\def\gold{2.6} 
			\node (V0) at (0-\offset:\radius) {$\N$};
			\node (E0) at (36-\offset:\radius/\gold) {$\N$};
			\node (V1) at (72-\offset:\radius) {$\N$};
			\node (E1) at (107-\offset:\radius/\gold) {$\N$};
			\node (V2) at (144-\offset:\radius) {$\N$};
			\node (E2) at (180-\offset:\radius/\gold) {$\N$};
			\node (V3) at (-144-\offset:\radius) {$\N$};
			\node (E3) at (-107-\offset:\radius/\gold) {$\N$};
			\node (V4) at (-72-\offset:\radius) {$\N$};
			\node (E4) at (-36-\offset:\radius/\gold) {$\N$};

			\path[->,font=\small]
			(V0) edge [color=red] node[auto,swap] {$\alpha$} (V1) 
			(V1) edge [color=red] node[auto,swap] {$\alpha$} (V2) 
			(V0) edge [color=red] node[auto] {$Id\star\alpha$} (V4) 
			(V4) edge [color=red] node[auto] {$\alpha$} (V3) 
			(V3) edge [color=red] node[auto] {$\alpha \star Id$} (V2) 
			
			(V0) edge [color=olive] node[auto,swap] {$\rho^{-1}$} (E0)
			
			(E0) edge [color=olive] node[auto,swap] {$ {\lambda}$} (V1)
			
			(V1) edge [color=olive] node[auto,swap] {$\rho^{-1}$} (E1)
			(E1) edge [color=olive] node[auto,swap] {$ {\lambda}$} (V2)
			
			(V0) edge [color=olive] node[auto] {$Id \star\rho^{-1}$} (E4)
			(E4) edge [color=olive] node[sloped, anchor=center, below] {$Id\star  {\lambda}$} (V4)
			
			(V4) edge [color=olive] node[auto] {$\rho^{-1}$} (E3)
			(E3) edge [color=olive] node[auto] {$ {\lambda}$} (V3)
			
			(V3) edge [color=olive] node[sloped, anchor=center, above] {$\rho^{-1}\star Id$} (E2)
			(E2) edge [color=olive] node[sloped, anchor=center, above] {$ {\lambda}\star Id$} (V2);
			
			\path[->,font=\small]
			(E0) edge [color=blue] node[sloped, anchor=center, above] {$\rho^{-1}\lambda$} (E1)
			(E1) edge [color=blue] node[sloped, anchor=center, above] {$(\lambda^{-1}\star Id)\lambda$} (E2)
			(E3) edge [color=blue] node[sloped, anchor=center, above] {$(\rho^{-1}\star Id)\lambda$} (E2)
			(E4) edge [color=blue] node[sloped, anchor=center, below] {$\rho^{-1}(Id\star \lambda)$} (E3)
			(E4) edge [color=blue] node[sloped, anchor=center, above] {$\rho^{-1}(Id\star \rho)$} (E0);
			\end{tikzpicture}
		}
	\end{center}
	\ \\
\end{figure}

\subsection{Interpreting the pentagram}
The interpretation of MacLane's pentagon (the {\color{red} red} paths) as the 1-skeleton of Stasheff's fourth associahedron is well-known; nodes correspond to vertices (i.e. binary well-bracketings of four symbols), and edges correspond to mappings between them. In Figure \ref{pentagram-fig}, the labels on blue paths (i.e. the inner pentagon) are the unique components of natural transformations between the following injective homomorphisms from $\SN^{\times 4}$ to $\SN$ : 
\[\{ \ \  (\_ \star (\_ \star \_ \star \_ )) \ , \ 
((\_ \star \_ \star \_ ) \star \_ ) \ , \ 
((\_ \star \_ ) \star \_ \star \_ ) \ , \ 
(\_ \star (\_ \star \_ ) \star \_ ) \ , \ 
(\_ \star \_ \star (\_  \star \_ )) \ \ \}
\]
We may of course interpret these homomorphisms as edges of $\mathcal K_4$, and the blue paths as mappings between edges of the fourth associahedron.

Finally, we need to give an account of the {\color{olive} green} paths. Again interpreting via the fourth associahedron $\mathcal K_4$, we should understand these as, {\em mappings between edges and vertices}; they are the components of natural transformations such as 
\[ (\_ \star (\_ \star (\_ \star \_ ))) \ \Rightarrow (\_ \star (\_ \star \_ \star \_ )) \ \ \ \ \mbox{ with unique component } \ \   Id \star \rho^{-1} \]
Equivalently, they correspond to deleting / inserting a matching pair of brackets in a string of four symbols. This then provides us with the interpretation of Collatz's bijection : it inserts a matching (right-associated) pair of brackets into such a string; the reduced Collatz bijection performs the same task for a left-associated pair. Their inverses delete such matching pairs of brackets.

\subsection{The OCC, and natural transformations between bracketings}The Original Collatz Conjecture concerns fixed-points of powers of $\rho$ --- implicitly, it is based on the homomorphism from $(\mathbb N , +)$ to $\SN$ given by $n\mapsto \rho^n$.  We consider $\SN$ as simply a distinguished subgroup of the symmetric inverse monoid $\IN$ and make the following definition :
\begin{definition}
	We define the {\bf left-} and {\bf right- Collatz homomorphisms} to be the monoid homomorphisms from $(\N,+)$ to $\IN$ given by
	\[ \Cz_L(n)=\lambda^n \ \ \mbox{ and } \ \ \Cz_R(n)= \rho^n \]
\end{definition}

\begin{proposition}There exists a natural transformation $\Cz_L\Rightarrow \Cz_R$  whose unique component is the successor function $succ\in \IN$.
\end{proposition}
\begin{proof}
	Let us re-write the key identity of Definition \ref{rcb-def} as $1+\left(\lambda\right)(n)=\rho(n+1)$, for all $n\in \N$.
	Then $1+\left(\lambda^K\right)(n)=\rho^K(n+1)$ for all $k\in \N$,  
	giving, for all $k\in \N$,
	\[ succ .\Cz_R(k )  \ = \ \Cz_L(k ). succ \ \ \ \forall n\in (\mathbb N,+) \]
	as required.
\end{proof}
Thus there exists a natural transformation between the left- and right- Collatz homomorphisms, whose unique component is simply the successor function.

\subsection{A general setting, and future directions}
It will not have escaped the reader's attention that the identity isomorphism, Girard's conjunction, and Definition \ref{star3-def} generalise to a countably infinite $\N^+$-indexed family of group homomorphisms.

\begin{definition}\label{stark-def}
	For all $k>0$, we define $\mu_{(k)}:\SN^{\times k}\rightarrow \SN$ to be the injective group homomorphism given by, for all $F=(f_0,f_1,\ldots , f_{k-1})\in \SN^{\times k}$,
	\[ \mu_{(k)}\left(F\right) (n) \ = \ k.f_{n \ (mod \ k)} \left( \frac{n - (n \ (mod \ k))}{k}\right)+ (n \ (mod \ k)) \]
	giving the identity as $Id_N=\mu_{(1)} : \SN\rightarrow \SN$, Girard's conjunction as $(\_ \star \_) =\mu_{(2)}:\SN\times \SN\rightarrow \SN$, and $(\_ \star \_ \star \_) = \mu_{(3)}:\SN^{\times 3}\rightarrow \SN$, \ldots
	
	Informally, $\mu_{(k)}\left(F\right)$ replicates the action of each member of $\{ f_0 , f_1,\ldots ,f_{k-1} \}$ on the corresponding  member of the exact covering system 
	$\{ k\mathbb N + j \}_{j=0 .. k-1}$.
\end{definition}

Our claim -- to be justified in \cite{PH22,PH22c} -- is that this indexed family of operations generates a sub-operad of $Endo(\SN)$ isomorphic to the (free, formal) operad $RPT$ of `rooted planar trees', and this observation allows us to give commuting diagrams of congruential functions between arbitrary facets of associahedra of all dimensions. 

In this setting, the original and reduced Collatz bijections $\rho,\lambda\in \SN$ are special, in that they label the {\em third} associahedron -- which simply consists of two vertices and an edge --  as follows : 
\[ 
\begin{tikzcd}
( ( \bullet  \bullet ) \bullet) &  & (  \bullet  \bullet   \bullet ) \ar{ll}[swap]{\lambda} \ar{rr}{\rho} &  & ( \bullet(  \bullet  \bullet  )  ) 
\end{tikzcd} \]
As we are also able to establish label-preserving embeddings $\mathcal K_a\hookrightarrow \mathcal K_{a+b}$, for all $a,b\in \mathbb N$, they may therefore be found in commuting diagrams over $\SN$ derived from arbitrary-dimensional associahedra.    

A fuller description of this is found in \cite{PH22,PH22c}, based on the underlying algebra described in \cite{PH22b}.





\section*{Acknowledgements}
Although no man is an island, I prefer to add acknowledgments -- of which there will be many -- to final published versions of papers.

\bibliography{OCCBib}

\appendix

\end{document}